\documentclass{proc-l}
\usepackage{bbm}
\usepackage{float, verbatim}
\usepackage{amsmath}
\usepackage{amssymb}
\usepackage{amsthm}
\usepackage{amsfonts}
\usepackage{graphicx}
\usepackage{mathtools}
\usepackage{enumitem}
\usepackage[pagebackref]{hyperref}
\usepackage{color}
\usepackage{tikz}
\usepackage{xcolor}
\usepackage{pgfplots}
\usepackage{caption}
\usepackage{subcaption}
\usepackage{enumerate}
\usepackage{url}

\hypersetup{
	colorlinks=true,
	linkcolor=blue,
	filecolor=magenta,      
	urlcolor=cyan,
	citecolor= red,
}

\usepackage[final]{showlabels}

\newcommand{\Haus}{\dim_{\mathrm{H}}}

\newtheorem*{thm*}{Theorem}
\newtheorem*{conj*}{Conjecture}
\newtheorem*{ques*}{Question}
\newtheorem*{rem*}{Remark}
\newtheorem*{defn*}{Definition}
\newtheorem*{mainques*}{Main questions}

\newtheorem{thm}{Theorem}[section]

\newtheorem{lma}[thm]{Lemma}

\newtheorem{defn}[thm]{Definition}

\newtheorem{prop}[thm]{Proposition}

\newtheorem{rem}[thm]{Remark}
\newtheorem{ques}[thm]{Question}

\def\supp{\mathrm{supp}}

\begin{document}
\title[Continuous Bernoulli convolutions]{Bernoulli convolutions with Garsia parameters in $(1,\sqrt{2}]$ have continuous density functions}

\author{Han Yu}
\address{Han Yu\\
	Department of Pure Mathematics and Mathematical Statistics\\University of Cambridge\\CB3 0WB \\ UK }
\curraddr{}
\email{hy351@cam.ac.uk}
\thanks{}

\subjclass[2020]{ 28A78, 42A85, 37A44}

\keywords{Mahler measure, Bernoulli convolution, Garsia number}

\maketitle

\begin{abstract}
Let $\lambda\in (1,\sqrt{2}]$ be an algebraic integer with Mahler measure $2.$ A classical result of Garsia shows that the Bernoulli convolution $\mu_\lambda$ is absolutely continuous with respect to the Lebesgue measure with a density function in $L^\infty$. In this paper, we show that the density function is continuous.
\end{abstract}

\maketitle
\allowdisplaybreaks

\section{Introduction}
Let $\lambda\in (1,2].$ Consider the random sum
\[
S_\lambda=\sum_{i\geq 0} \pm \lambda^{-i},
\]
with each $\pm$ being chosen independently with equal probability, i.e. $(1/2,1/2).$ Let $\mu_\lambda$ be the distribution of $S_\lambda.$ It is the Bernoulli convolution with parameter $\lambda.$ Notice that when $\lambda=2,$ $\mu_\lambda$ is extremely simple. It is simply the Lebesgue measure on $[-2,2]$ which is normalized to be a probability measure. In general, it is not an easy task to understand $\mu_\lambda.$ The study of $\mu_\lambda$ has a long and rich history. It is a subject that mixes algebraic number theory, probability theory, harmonic analysis, geometric measure theory and additive combinatorics. 

The ultimate problem in this area is to completely understand $\mu_\lambda$ for every $\lambda\in (1,2].$ A general outline can be viewed as follows. Of course, many more questions can be asked.
\begin{itemize}
    \item For each $\lambda\in (1,2],$ determine whether or not $\mu_\lambda$ is absolutely continuous with respect to the Lebesgue measure.
    
    \item If $\mu_\lambda$ is not absolutely continuous with respect to the Lebesgue measure, then determine its Hausdorff dimension. 
    
    \item If $\mu_\lambda$ is absolutely continuous, then determine its density function $f_\lambda.$ Clearly, $f_\lambda$ is $L^1.$ 
    
    \item Determine whether or not $f_\lambda\in L^p,p>1.$
    
    \item Determine whether or not $f_\lambda\in L^\infty.$
    
    \item Determine whether or not $f_\lambda$ is continuous (and differentiable of any order). 
    
    \item If $f_\lambda$ is continuous, determine the value of $f_\lambda(x)$ for each $x\in\mathbb{R}.$
\end{itemize}

Up to now, we have the following answers to the above questions.
\subsection*{Hausdorff dimension} (See Section \ref{sec: Haus})
\begin{itemize}
    \item \cite{H14}, \cite{BV}, \cite{BV2}, \cite{VarjuDimension}: If $\Haus \mu_\lambda<1$ then $\lambda$ is a root of infinitely many polynomials with coefficients $\pm 1, 0.$

\item \cite{AFKP}: There is an effective algorithm to approximate $\Haus \mu_\lambda$ for all algebraic $\lambda.$\footnote{In fact, there is an effective algorithm to compute the Lyapunov exponent $h_\lambda$ of $\mu_\lambda.$ It is known that $\Haus \mu_\lambda=\min\{1,h_\lambda/\log \lambda\}.$ There is only one issue: if $h_\lambda=\log \lambda$ then the algorithm cannot be used to confirm this. In this paper, $\log$ is the logarithm with base $2$.} 
\end{itemize}

\subsection*{Continuity}
Let $\lambda$ be an algebraic number over $\mathbb{Q}.$ The Mahler measure of $\lambda$ is 
\[
M_\lambda=|a|\prod_{i=1}^n \max\{|\lambda_i|,1\},
\]
where $a$ is the leading coefficient of the minimal polynomial of $\lambda$ and $\lambda_1,\dots,\lambda_n$ are all the conjugates of $\lambda.$
\begin{itemize}
    \item \cite{Garsia}: If $\lambda$ is a Garsia number, i.e., an algebraic integer in $(1,2]$ with $M_\lambda=2$, then $\mu_\lambda$ is in $L^\infty.$ See Section \ref{Garsia Pre}.
    
    \item \cite{VarjuContinuity}: If $\lambda$ is an algebraic number in $(1,2]$ with 
	\[
	\lambda>1-c\min\{\log M_\lambda, (\log M_\lambda)^{-1-\epsilon}\},
	\]
	where $\epsilon, c$ are effectively computable constants, then $\mu_\lambda$ is absolutely continuous with a density function in $L(\log L).$
	
    \item \cite{K20}: Let $\lambda$ be an algebraic number in $(1,2]$ with
	\[
	(\log M_\lambda-\log 2)(\log M_\lambda)^2<\frac{1}{27}(\log M_\lambda-\log\lambda^{-1})^3\lambda^4.
	\]
	Suppose that $\lambda$ is not a root of polynomials with coefficients in $\pm 1,0.$ Then $\mu_\lambda$ is absolutely continuous.
\end{itemize}

\subsection*{Generic results}
\cite{E}, \cite{K71}, \cite{K85}, \cite{Solomyak}, \cite{PS}, \cite{ShD}, \cite{Sh}: There is a set $E\subset (1,2]$ with zero Hausdorff dimension such that for each $p\in (1,\infty),$ for each $\lambda\in (1,2]\setminus E,$ $\mu_\lambda$ is in $L^p.$ Moreover, if $\lambda\in (1,\sqrt{2}]\setminus E,$ then $\mu_\lambda$ has a continuous density function.
\\

It is relatively less known when $\mu_\lambda$ is absolutely continuous with respect to the Lebesgue measure with a continuous density function. In this situation, we say that $\mu_\lambda$ is continuous.  By using the fact that convolutions of $L^2$ functions are continuous, it is possible to find such examples of $\mu_\lambda$. See Section \ref{sec: pre example}. Our main result is as follows.
\begin{thm}
Let $\lambda\in (1,\sqrt{2}]$ be a Garsia number. Then $\mu_\lambda$ is continuous. In fact, for some $\beta>0,$ $\mu_\lambda$ is $\beta$-H\"{o}lder continuous.
\end{thm}
There is no lack of such algebraic numbers. See \cite{Hare}. A particular example is the root of $x^5-2x^4+2x^3-2x^2+2x-2=0$ in $(1,\sqrt{2}].$ The key point is that it is not clear whether or not $\mu_\lambda$ is a convolution of $L^2$ functions.

The proof of the above theorem is simple, although the building blocks are rather heavy. It utilises a result of  \cite[Theorem 1.6]{DFW} on the power Fourier decay and a result of \cite[Theorem 6.2]{Sh} on the regularity of self-similar measures with the exponential separation condition. (This is a refinement of an earlier work of  \cite{H14}.) In fact, it is possible to show the following slightly more general result. In what follows, for functions $f,g:\mathbb{R}\to\mathbb{R}$, $f\ll g$ means that for a constant $C>0,$ $|f(x)|\leq C|g(x)|$ holds for all $x\in\mathbb{R}$ with $|x|>C.$
\begin{thm}\label{thm: main}
Let $\lambda\in (1,\sqrt{2}]$ be an algebraic number such that 
\[
|\hat{\mu}_\lambda(\xi)|\ll |\xi|^{-\sigma}
\]
for some $\sigma>0$ as $\xi\to\infty.$ If $\lambda$ is not a root of polynomials with coefficients $\pm 1,0,$ then $\mu_\lambda$ is $\beta$-H\"{o}lder continuous for some $\beta>0.$
\end{thm}
It is likely to be very difficult to test the power Fourier decay property for any specific given $\lambda$. Up to now, the only known examples are Garsia numbers. On the other hand, the power Fourier decay property holds 'generically'. The set of $\lambda$ such that $\mu_\lambda$ does not have power Fourier decay is extremely thin. It has Hausdorff dimension zero (due to Erd\H{o}s \cite{E} and Kahane \cite{K71}. See also \cite{ShD}).

We end the introduction with a small remark about the method. 
\begin{rem}
If we assume that $\mu_{\lambda^k},k\geq 1$ all have power Fourier decay, then Theorem \ref{thm: main} would be a direct consequence of \cite[Lemma 2.1]{ShD}. For Garsia parameters, we strongly believe that this should be true. However, we only know that $\mu_\lambda$ has power Fourier decay. This difficulty can be bypassed by using a weaker power decay property. See (\ref{virtual decay}) in the proof of Theorem \ref{thm: main}.
\end{rem}

\section{Preliminaries}
\subsection{Garsia number}\label{Garsia Pre}
\begin{defn}
 A Garsia number is an algebraic integer in $(1,2]$ with Mahler measure $M_\lambda=2$. Equivalently, it is an algebraic integer in $(1,2]$ with all of its conjugates lying strictly outside of the unit circle.
\end{defn}
The equivalence in the definition is due to Garsia. See \cite[Lemma 1.7]{Garsia}. Moreover, if $\lambda$ is Garsia, then its minimal polynomial $P_\lambda$ has a constant term $\pm 2.$ 
 
\subsection{Self-similar sets/measures on $\mathbb{R}$ and separation conditions}
Let $K\geq 2$ be an integer. Let $r_1,\dots,r_K\in (0,1)$ and $a_1,\dots,a_K\in\mathbb{R}.$ Consider the set of affine maps
\[
\mathcal{F}=\{x\in\mathbb{R}\to r_ix+a_i\}_{1\leq i\leq K}.
\]
By a result in \cite{H81}, there is a uniquely determined compact subset $F\subset\mathbb{R}$ with the property that
\[
F=\bigcup_{i=1}^K f_i(F).
\]
This is the self-similar set associated with the self-similar system $\mathcal{F}.$ Let $p_1,\dots,p_K$ be positive numbers with $p_1+\dots+p_K=1.$ Then again by a result in \cite{H81}, there is a uniquely determined Borel probability measure $\mu$ supported on $F$ with the property that
\[
\mu=\sum_{i=1}^K p_i (f_i\mu).
\]
This is the self-similar measure assiociated with the self-similar system $\mathcal{F}$ and probability weights $p_1,\dots,p_K.$

If $r_1=\dots=r_K$, we say that the $\mathcal{F}, F, \mu$ are homogeneous. If $K=2,$ $r_1=r_2=\lambda^{-1}\in (0,1),a_1=1,a_2=-1, p_1=p_2=1/2,$ then we denote $\mu_\lambda$ to be the associated self-similar measure. This is called to be the Bernoulli convolution with parameter $\lambda.$ This is equivalent to the random sum definition at the beginning of this paper.

We say that $\mathcal{F}'$ is a subsystem of $\mathcal{F}$ if for some $n\geq 1,$
\[
\mathcal{F}'\subset\{f_{\mathbf{i}}\}_{\mathbf{i}\in\{1,\dots,K\}^n},
\]
where $f_{\mathbf{i}}$ is the iterated function
\[
x\in\mathbb{R}\to f_{i_n}\circ f_{i_{n-1}}\circ\dots\circ f_{i_1}(x).
\]

We now introduce the exponential separation condition. Let $n\geq 1$ be an integer. Consider the set of points
\[
F_n=\{f_{\mathbf{i}}(0)\}_{i\in \{1,\dots,K\}^n}.
\]
For $\mathbf{i},\mathbf{j}\in\{1,\dots,K\}^n,$ define the distance
\[
d_n(i,j)=|f_{\mathbf{i}}(0)-f_{\mathbf{j}}(0)|
\]
if $r_{i_1}\dots r_{i_n}=r_{j_1}\dots r_{j_n}$ and $d_n(i,j)=\infty$ otherwise. The $n$-th level gap is defined to be
\[
\Delta_n=\min\{d_n(\mathbf{i},\mathbf{j}),\mathbf{i}\neq \mathbf{j}, \mathbf{i},\mathbf{j}\in\{1,\dots,K\}^n\}.
\]
We say that $\mathcal{F}, F, \mu$ has exact overlaps if $\Delta_n=0$ for infinitely many $n.$ Assume that there are no exact overlaps. We say that $\mathcal{F}, F, \mu$ has the exponential separation condition if
\[
\liminf_{n\to\infty}|\log \Delta_n|/n<\infty.
\]
Suppose that all the parameters in $\mathcal{F}$ are algebraic. If moreover $\mathcal{F}$ does not have exact overlaps, then it has the exponential separation condition. See \cite[Section 5.5]{H14}.

\subsection{Hausdorff dimension of measures}\label{sec: Haus}
Let $\mu$ be a Borel probability measure on $\mathbb{R}.$ The lower/upper Hausdorff dimensions of $\mu$ are defined to be
\[
\underline{\Haus}\mu=\inf\{\Haus B: B \text{ is a Borel subset of }\mathbb{R}, \mu(B)>0\}
\]
\[
\overline{\Haus}\mu=\inf\{\Haus B: B \text{ is a Borel subset of }\mathbb{R}, \mu(B)=1\},
\]
where $\Haus B$ is the Hausdorff dimension of $B$. See \cite[Page 14]{Ma2}. It can be checked (e.g. \cite[Theorem 2.3]{Heu}) that
\[
\underline{\Haus}\mu=\mathrm{essinf}_{x\sim \mu}\liminf_{\delta\to 0}\frac{\log \mu(B_\delta(x))}{\log \delta}
\]
\[
\overline{\Haus}\mu=\mathrm{esssup}_{x\sim \mu}\liminf_{\delta\to 0}\frac{\log \mu(B_\delta(x))}{\log \delta}.
\]
Here $B_\delta(x)$ is the metric ball of radius $\delta$ centred at $x.$ For Bernoulli convolutions and more generally, self-similar measures, the Hausdorff dimension is always well defined (this property is also known as the exact dimensionality)
\[
\Haus \mu=\underline{\Haus}\mu=\overline{\Haus}\mu=\lim_{\delta\to 0}\frac{\log \mu(B_\delta(x))}{\log \delta}
\]
for $\mu.a.e$ $x$. See \cite{Feng}. 
\subsection{Fourier transform and power Fourier decay}
Let $\mu$ be a Borel probability measure on $\mathbb{R}$. The Fourier transform of $\mu$ is defined as follows, 
\[
\hat{\mu}(\xi)=\int_{\mathbb{R}} e^{-2\pi i x\xi}d\mu(x).
\]
The following result is proved in \cite[Theorem 1.6]{DFW}.
\begin{thm}[Dai, Feng and Wang]
Let $\lambda$ be a Garsia number. Then the Bernoulli convolution $\mu_\lambda$ has power Fourier decay, i.e.
\[
|\hat{\mu}_\lambda(\xi)|\ll |\xi|^{-\sigma}
\]
for some $\sigma>0.$
\end{thm}
\subsection{Shmerkin's regularity theorem}
For each $\alpha<1$, we say that $\mu_\lambda$ is $\alpha$-regular if there are numbers $C, r_0>0$ such that for all $r<r_0$ and $x\in\mathbb{R},$
\[
\mu_\lambda(B_r(x))\leq C r^{\alpha}.
\]
If $\mu_\lambda$ is $\alpha$-regular for each $\alpha<1,$ we say that $\mu_\lambda$ is $1_-$-regular. According to \cite[Theorem 6.2]{Sh}, $\mu_\lambda$ is $1_-$-regular if $\mu_\lambda$ has the exponential separation condition. This is the case when $\lambda$ is algebraic and it is not a root of polynomials with coefficients $\pm 1,0.$ The $\alpha$-regularity of measures is closely related to the average $L^2$ Fourier decay. Let $\alpha\in (0,1).$ Let $\mu$ be an $\alpha$-regular probability measure on $\mathbb{R}.$ It is known that (\cite[Theorem 3.10]{Ma2}, \cite[Page 19, Discussion after the proof of Theorem 2.9]{Ma2})
\[
\int |\xi|^{-(1-s)} |\hat{\mu}(\xi)|^2d\xi <\infty
\]
for all $s<\alpha.$ In particular, if $\mu$ is $1_-$-regular, then the above holds with $1-s$ being replaced by any $\epsilon>0.$

The above regularity result holds more generally for all homogeneous self-similar measures\footnote{An earlier result of Hochman \cite{H14} provides a weaker regularity result which holds for all self-similar measures, not only the homogeneous ones.}. In particular, let $r^{-1}=\lambda\in (1,\infty).$ Let $K\geq 2$ be an integer. Let $a_1,\dots,a_K\in\mathbb{R}.$ Let $p_1,\dots,p_K>0$ with $p_1=\dots=p_k=1/K.$ For the associated self-similar measure $\mu$, we define its self-similarity dimension to be
\[
\dim_s \mu=\frac{\log K}{\log \lambda}.
\]
Now assume that $\lambda, a_1,\dots,a_K$ are algebraic numbers. As long as there are no exact overlaps, \cite[Theorem 6.2]{Sh} says $\mu$ is $\alpha$-regular for each
\[
\alpha<\min\{1,\dim_s\mu\}.
\]
\subsection{Continuous Bernoulli convolutions: some known examples}\label{sec: pre example}
Here we record some specific examples for continuous Bernoulli convolutions. Let $\lambda$ be such that $\lambda^k$ for some $k\geq 2$ is a Garsia number. Then $\mu_{\lambda}$ is continuous because it is a convolution of measures that have $L^\infty$ density functions, i.e. $\mu_\lambda=*_{j=0}^{k-1}T_{\lambda^j}\mu_{\lambda^k},$ where $T_{\lambda^j}$ is the linear map $x\to \lambda^j x.$ Examples of this kind include $2^{1/n},n\geq 2.$ A slightly less trivial example is the root of polynomial $x^{10}-2x^8+2x^6-2x^4+2x^2-2$ in $(1,2]$. In fact, let $\lambda$ be this root. Then we see that $\lambda^2$ is the root of $x^5-2x^4+2x^3-2x^2+2x-2$ in $(1,\sqrt{2}]$. Thus $\lambda^2$ is a Garsia number. 

Not all Garsia numbers are radical roots of other Garsia numbers. For example, consider the number $s=\lambda^2$ as above. Let $n\geq 2$ be an integer. Suppose that $s^n$ is Garsia. Let $P(x)$ be the minimal polynomial of $s^n.$ Then $P$ is monic (i.e. the leading term has coefficient $1$) and $P(s^n)=0.$ Let us look at the polynomial $Q(x)=P(x^n).$ We see that $Q(s)=0.$ Thus $Q(x)$ is a multiple of the polynomial $T(x)=x^5-2x^4+2x^3-2x^2+2x-2.$ Notice that $Q(0)=P(0)$ is in $\pm 2.$ This is because $s^n$ is Garsia. Thus we see that $Q(x)/T(x)$ must have constant term $\pm 1$. Suppose that $h(x)=Q(x)/T(x)$ is not a constant polynomial. Then it must be monic. Let $x_0$ be a root of $h(x).$ Then $P(x^n_0)=0.$ Therefore $|x_0|>1$. This is because $s^n$ is Garsia and therefore all its conjugates are outside of the unit circle. See Section \ref{Garsia Pre}. Thus all roots of $h$ are outside of the unit circle. Observe that $|h(0)|$ is the product of the absolute values of all the roots of $h$. This implies that $|h(0)|>1.$ This contradiction shows that $h(x)$ must be a constant polynomial, i.e. $h(x)$ is $\pm 1.$ Therefore we have
\[
P(x^n)=T(x) \text{ or } P(x^n)=-T(x).
\]
This is impossible unless $n=1$.

\section{proofs of the results}

In Section \ref{Garsia Pre}, we see that if $\lambda$ is a Garsia number, then $\lambda$ cannot be a root of polynomials with coefficients $\pm 1,0.$ This is because the minimal polynomial $P_\lambda$ has a constant term $\pm 2.$ This implies that the self-similar measure $\mu_\lambda$ has no exact overlaps. Since $\lambda$ is algebraic, $\mu_\lambda$, this implies that $\mu_\lambda$ satisfies the exponential separation condition.

The following result is an $L^1$-analogy of \cite[Lemma 2.1]{ShD}.  Technically we do not use this lemma. It follows directly from  \cite[Lemma 2.1]{ShD} and the fact that convolutions of $L^2$ functions are continuous. However, we do need this lemma to illustrate a simple idea involving Littlewood-Paley decomposition that will be applied later.

\begin{lma}\label{lemma: L2L2L2}
Let $\nu$ be a compactly supported Borel probability measure with power Fourier decay, i.e. for some $\sigma>0$, $|\hat{\nu}(\xi)|\ll |\xi|^{-\sigma}$. Let $\mu_1,\mu_2$ be compactly supported Borel probability measures which are $1_-$-regular. Then $\nu*\mu_1*\mu_2$ is continuous, i.e. it is absolutely continuous with respect to the Lebesgue measure, and the density function is continuous. In fact, the density function is $\beta$-H\"{o}lder continuous for $0<\beta<\sigma/(2+\sigma).$
\end{lma}
\begin{proof}
From the regularity of $\mu_1$ we see that for each $\epsilon>0,$
\[
\int |\xi|^{-\epsilon} |\hat{\mu}_1(\xi)|^2d\xi<\infty.
\]
Let $S_k=\{\xi: |\xi|\in [2^k,2^{k+1}]\}.$ We see that
\[
\int_{S_k}  |\hat{\mu}_1(\xi)|^2d\xi\ll 2^{\epsilon k}.
\]
A similar bound holds for $\mu_2.$ We choose $\epsilon<\sigma.$ This shows that for each integer $K\geq 1,$ we have

\begin{align*}
& \int_{|\xi|>2^{K}} |\hat{\nu}(\xi)\hat{\mu}_1(\xi)\hat{\mu}_2(\xi)|d\xi\\
&\ll \sum_{k\geq K} 2^{-\sigma k}\left(\int_{S_k} |\hat{\mu}_1(\xi)|^2d\xi \int_{S_k} |\hat{\mu}_2(\xi)|^2d\xi\right)^{1/2}\ll \sum_{k\geq K}2^{-\sigma k}2^{\epsilon k}\ll 2^{-(\sigma-\epsilon)K}.
\end{align*}
Thus the Fourier transform of $\mu'=\nu*\mu_1*\mu_2$ is in $L^1.$ From here, we see that $\mu'$ is absolutely continuous with a continuous density function. 

From now on, we identify $\mu'$ with its continuous density function. We now show that $\mu'$ is H\"{o}lder continuous via Littlewood-Paley decomposition. This argument is standard. See \cite[Theorem 6.3.7]{Gra}. We provide all the details for the sake of self-containedess. Since $\mu'$ has absolutely integrable Fourier transform, we see that for each $x\in\mathbb{R},$
\[
\mu'(x)=\int \hat{\mu'}(\xi)e^{2\pi i \xi x}dx.
\]
Let $h>0.$ We see that
\[
\mu'(x+h)=\int \hat{\mu'}(\xi)e^{2\pi i \xi x}e^{2\pi i \xi h}dx.
\]
Thus we have for each $R>0,$ $\epsilon\in (0,\sigma),$
\begin{align*}
|\mu'(x+h)-\mu'(x)|&=\left|\int \hat{\mu}'(\xi)e^{2\pi i \xi x}(1-e^{2\pi i \xi h})d\xi\right|\\
& \leq \int_{|\xi|\leq R}|\hat{\mu'}(\xi)||e^{2\pi i \xi h}-1|d\xi+2\int_{|\xi|>R}|\hat{\mu'}(\xi)|d\xi \\
&\ll \int_{|\xi|\leq R}|\hat{\mu'}(\xi)|\min\{2\pi |\xi h|,2\}d\xi+R^{-(\sigma-\epsilon)}\\
&\ll |h|R^{2}+R^{-(\sigma-\epsilon)}.
\end{align*}
Choosing $R=|h|^{-1/(2+\sigma-\epsilon)}$, we see that
\[
|\mu'(x+h)-\mu'(x)|\ll |h|^{(\sigma-\epsilon)/(2+\sigma-\epsilon)}.
\]
This concludes the proof.
\end{proof}

Now we can extract information from Bernoulli convolutions. 
\begin{proof}[Proof of Theorem \ref{thm: main}]
Let $\lambda\in (1,\sqrt{2}]$ be an algebraic number which is not a root of polynomials with coefficients $\pm 1,0.$ 

Let $k\geq 2$ be an integer. Let $J\subset \{0,\dots, k-1\}$ be a nonempty subset. Consider the measure
\[
\nu_J=*_{j\in J} T_{\lambda^{j}}\mu_{\lambda^k},
\]
where $T_{\lambda^{j}}$ is the linear map $x\to \lambda^{j}x.$
Notice that $J=\{0,\dots,k-1\}$ implies that $\nu_J=\mu_\lambda.$ 

The measure $\nu_J$ is a self-similar measure with contraction ratio $\lambda^k$ and $2^{\# J}$ many translations. It is a subsystem of $\mu_\lambda$, and thus there are no exact overlaps in $\nu_J.$ Since $\lambda$ is algebraic, by \cite[Theorem 6.2]{Sh}, we see that $\nu_J$ is $\alpha$-regular for each
\[
\alpha<\min\left\{\frac{\log 2^{\# J}}{\log \lambda^{k}},1\right\}.
\]

Suppose that 
\begin{align*}\label{power decay}
\hat{\mu}_{\lambda}(\xi)\ll|\xi|^{-\sigma}\tag{power decay}
\end{align*}
for some $\sigma>0.$ It would be nice if for each $k\geq 1$ there is a $\sigma_k>0$ such that
\[
\hat{\mu}_{\lambda^k}(\xi)\ll|\xi|^{-\sigma_k}.
\]
This seems to be too much to hope. To get around this issue, we use a weaker alternative:
\begin{itemize}
    \item Let $k\geq 2$ be an integer. Then for $\xi\to\infty,$ we have
\begin{align*}\label{virtual decay}
\min_{j\in\{0,\dots,k-1\}}\left\{|\hat{\mu}_{\lambda^k}(\lambda^j \xi)|\right\}\ll |\xi|^{-\sigma/k}.\tag{virtual decay}
\end{align*}
\end{itemize}
To see why (\ref{virtual decay}) holds for $\mu_\lambda$, observe that
\[
\mu_\lambda=*_{j\in\{0,\dots,k-1\}}T_{\lambda^j}\mu_{\lambda^k}.
\]
This implies that for each $\xi\in\mathbb{R},$
\[
|\hat{\mu_\lambda}(\xi)|=\prod_{j=0}^{k-1}|\hat{\mu}_{\lambda^k}(\lambda^j\xi)|.
\]
Because of (\ref{power decay}), we see that
\[
\prod_{j=0}^{k-1}|\hat{\mu}_{\lambda^k}(\lambda^j\xi)|\ll |\xi|^{-\sigma}.
\]
Now (\ref{virtual decay}) follows because
\[
\left(\min_{j\in\{0,\dots,k-1\}}\left\{|\hat{\mu}_{\lambda^k}(\lambda^j \xi)|\right\}\right)^{k}\leq\prod_{j=0}^{k-1}|\hat{\mu}_{\lambda^k}(\lambda^j\xi)|.
\]
Let $k$ be a large integer which will be chosen later. Observe that
\[
\int |\hat{\mu}_\lambda(\xi)| d\xi=\int \prod_{j=0}^{k-1}|\hat{\mu}_{\lambda^k}(\lambda^j \xi)|d\xi.
\]
Because of the property (\ref{virtual decay}), we see that
\[
\prod_{j=0}^{k-1}|\hat{\mu}_{\lambda^k}(\lambda^j \xi)|\ll |\xi|^{-\sigma/k} \sum_{l=0}^{k-1} \prod_{j=0,j\neq l}^{k-1}|\hat{\mu}_{\lambda^k}(\lambda^j \xi)|.
\]

For each $l\in\{0,\dots,k-1\},$ let
\[
\nu_l=*^{j=k-1}_{j=0,j\neq l}T_{\lambda^j}\mu_{\lambda^k}.
\]
From now on, we fix $l=0$, but the following arguments work for all other $l$ as well. Let $J\subset\{1,\dots,k-1\}$ be a subset. Let $J'$ be the complement. We assume that $k-1$ is an even number and $\# J=\# J'=(k-1)/2.$ We split $\nu_0$ into two parts
\[
\nu_0=\nu_J*\nu_{J'}.
\]
Notice that $\nu_J, \nu_{J'},$ are all $\alpha$-regular for 
\[
\alpha<\min\left\{\frac{\log 2^{(k-1)/2}}{\log \lambda^k},1\right\}.
\]
Observe that
\[
\frac{\log 2^{(k-1)/2}}{\log \lambda^k}=\frac{k-1}{2k}\frac{\log 2}{\log \lambda}.
\]
Thus as long as $\log 2/\log \lambda>2$, i.e. $\lambda<\sqrt{2},$ it is possible to choose $k$ such that
\[
\frac{k-1}{2k}\frac{\log 2}{\log \lambda}>1.
\]
We fix such a number $k.$ For this number $k,$ we see that $\nu_{J},\nu_{J'}$ are $1_-$-regular. Observe that
\begin{align*}
&\int |\xi|^{-\sigma/k} |\hat{\nu}_0(\xi)|d\xi=\int |\xi|^{-\sigma/k} |\hat{\nu}_J(\xi)||\hat{\nu}_{J'}(\xi)|d\xi.
\end{align*}
Since $\nu_J, \nu_{J'}$ are $1_{-}$-regular, this is the same situation we had in Lemma \ref{lemma: L2L2L2}. With the same method, we see that
\[
\int |\xi|^{-\sigma/k} |\hat{\nu}_0(\xi)|d\xi<\infty.
\]
Arguing with other values for $l$ (there are finitely many of them) we see that
\[
\int \prod_{j=0}^{k-1}|\hat{\mu}_{\lambda^k}(\lambda^j \xi)|d\xi\ll \int |\xi|^{-\sigma/k} \sum_{l=0}^{k-1} \prod_{j=0,j\neq l}^{k-1}|\hat{\mu}_{\lambda^k}(\lambda^j \xi)|d\xi<\infty.
\]
From here we conclude that $\mu_\lambda$ has $L^1$ Fourier transform and thus $\mu_\lambda$ is absolutely continuous with a continuous density function. Moreover, we can also follow the last step of the proof of Lemma \ref{lemma: L2L2L2} we see that $\mu_\lambda$ is in fact $\beta$-H\"{o}lder for each positive $\beta$ with
\[
\beta<\frac{\sigma}{2k+\sigma}.
\]
\end{proof}
\section{Questions}
In this section, we pose two questions. Notice that if $\mu_\lambda$ is continuous, then it makes sense to take the value $\mu(0).$ 
\begin{ques}
Let $\lambda\in (1,\sqrt{2}]$ be a Garsia number. What is the value of $\mu_\lambda(0)$? Is it rational?
\end{ques}
Let $n\geq 1$ be an integer. If $\lambda=2^{1/n}$ we see that $\mu_\lambda$ is a convolution of Lebesgue measures on $n$ different intervals. In this case, the function $\mu_\lambda$ can be explicitly determined. Other than those examples, we are not aware of any other values $\lambda$ for which the above question is answered. 

Next, as we showed that $\mu_\lambda$ is $\beta$-H\"{o}lder continuous for some positive parameter $\beta$. It is natural to consider how this result can be improved.
\begin{ques}
Let $\lambda\in (1,\sqrt{2}]$ be a Garsia number. Is $\mu_\lambda$ a Lipschitz function? Is $\mu_\lambda$ almost everywhere differentiable? 
\end{ques}
Again, roots of $2$ provide us with examples for which this question can be answered.

\section{Acknowledgement}
HY was financially supported by the University of Cambridge and the Corpus Christi College, Cambridge. HY has received funding from the European Research Council (ERC) under the European Union's Horizon 2020 research and innovation programme (grant agreement No. 803711). HY thanks P. Shmerkin and P. Varj\'{u} for many very helpful comments.

\bibliographystyle{amsplain}

\end{document}